\documentclass[10pt,namelimits,sumlimits]{amsart}
\usepackage{appendix}
\usepackage{amssymb,amsmath, mathrsfs}
\usepackage[mathscr]{eucal}
\usepackage{color}
\usepackage{enumitem}
\usepackage[utf8]{inputenc}
\usepackage[leftcaption]{sidecap}
\usepackage{tikz}
 \usepackage{tikz-cd}
 \usepackage{stmaryrd}
\usetikzlibrary{matrix,arrows,decorations.pathmorphing}
\usetikzlibrary{positioning}
\usetikzlibrary{matrix,arrows,decorations.pathmorphing, shapes.geometric}

\usepackage{psfrag}

\usepackage{epstopdf}
\usepackage{graphicx}
\usepackage{verbatim}
\usepackage{amsmath,amsthm}
\usepackage{graphics}
\usepackage{color}
\usepackage{epsfig}
\usepackage{amssymb,amsmath}
\usepackage[mathscr]{eucal}
\usepackage{ upgreek }
\usepackage{tensor}
\usepackage{physics}

\newtheorem{theorem}[equation]{Theorem}

\newtheorem{lemma}[equation]{Lemma}

\newtheorem{corollary}[equation]{Corollary}

\newtheorem{definition}[equation]{Definition}

\theoremstyle{remark}
\newtheorem{remark}[equation]{Remark}
\newtheorem{notation}[equation]{Notation}

\newtheorem{assumption}[equation]{Assumption}

\numberwithin{equation}{section}

\newcommand{\R}{\mathbb{R}}

\newcommand{\N}{\mathbb{N}}
\newcommand{\cat}{\mathbb{K}}
\newcommand{\Lcal}{\mathcal{L}}
\newcommand{\Sph}{\mathbb{S}}
\newcommand{\Wcal}{\mathcal{W}}

\newcommand{\Psibold}{{\boldsymbol{\Psi}}}

\newcommand{\arccosh}{\operatorname{arccosh}}
\newcommand{\dbold}{{\mathbf{d}}}
\newcommand{\phicat}{\varphi_{\mathrm{cat}}}
\newcommand{\Phip}{\Phi'}
\newcommand{\psicut}{\psi_{\mathrm{cut}}}
\newcommand{\phigl}{\varphi_{\mathrm{gl}}}
\newcommand{\varphiunder}{{\underline{\varphi}}}
\newcommand{\varphiover}{\overline{\varphi}}
\newcommand{\zz}{\ensuremath{\mathrm{z}}}
\newcommand{\rr}{\ensuremath{\mathrm{r}}}
\newcommand{\sss}{{{\ensuremath{\mathrm{s}}}}}
\newcommand{\sunder}{\underline{\sss}}
\newcommand{\Gcat}{\cat_{p, \tau}}
\newcommand{\Gld}{\Gamma^{\phigl}_{\Omega}}
\newcommand{\Siginit}{\Sigma_{m, \tau}}
\newcommand{\Vol}{\operatorname{Vol}}
\newcommand{\Area}{\operatorname{Area}}

\begin{document}

\title[On the Canham Problem]{On the Canham Problem: Bending Energy\\Minimizers for any Genus and Isoperimetric Ratio}
\author[R.~Kusner]{Robert~Kusner}
\author[P.~McGrath]{Peter~McGrath}
\date{}
\address{Department of Mathematics, University of Massachusetts,
Amherst, MA, 01003} \email{profkusner@gmail.com, kusner@math.umass.edu}
\address{Department of Mathematics, North Carolina State University, Raleigh NC 27695} 
\email{pjmcgrat@ncsu.edu}
\maketitle

\begin{abstract}
Building on work of Mondino-Scharrer, we show that among closed, smoothly embedded surfaces in $\R^3$ of genus $g$ and given isoperimetric ratio $v$, there exists one with minimum bending energy $\Wcal$.  We do this by gluing $g+1$ small catenoidal bridges to the bigraph of a singular solution for the linearized Willmore equation $\Delta (\Delta +2)\varphi=0$ on the $(g+1)$-punctured sphere $\Sph^2$ to construct a comparison surface of genus $g$ with arbitrarily small isoperimetric ratio $v\in (0, 1)$ and $\Wcal < 8\pi$. 
\end{abstract}

\section{Introduction}
\label{S:intro}
Motivated by the problem of explaining the shapes of biophysical membranes like red blood cells \cite{Canham} or phospholipid vesicles \cite{Helfrich} variationally, 
significant progress has been made \cite{Schygulla, KellerMondino, MondinoScharrer, Scharrer}, using previous work of \cite{LiYau, Kusner89, Simon, Kusner96, Bauer}, on the following basic mathematical question:
\begin{center}
\emph{Among closed, smoothly embedded surfaces in Euclidean three-space of a given genus\\ and prescribed isoperimetric ratio, is there a surface with minimum bending energy?} 
\end{center}

\noindent
For a smoothly embedded surface $S\subset\R^3$, its isoperimetric ratio $v(S)$ and its bending energy $\Wcal(S)$ are defined by
\begin{align}
\label{EW}
v(S):=36\pi\frac{\Vol(S)^2}{\Area(S)^3}, \quad \quad
\Wcal(S):= \frac{1}{4}\int_{S} H^2 d \sigma,
\end{align}
where $H$ and $d\sigma$ are the mean curvature and area element of $S$, and where $\Vol(S)$ is the volume of the domain in $\R^3$ that $S$ bounds. 
In this paper, we answer the above question affirmatively:

\begin{theorem}
\label{TMain}
Let $g \in \{0, 1, 2, \dots \}$ and $v \in (0, 1)$.  Among all closed, smoothly embedded surfaces $S \subset \R^3$ of genus $g$ and $v(S) = v$, there exists $\underline{S}_{g,v}$ with minimum bending energy $\Wcal(\underline{S}_{g,v}) < 8\pi$.
\end{theorem}

\noindent
Our main contribution is to construct, for each $g$ and prescribed $v\in (0, 1)$, a smooth embedded surface ${S}_{g,v}\subset \R^3$ of genus $g$, isoperimetric ratio $v$, and bending energy $\Wcal({S}_{g,v})<8\pi$. From this, the conclusion of Theorem \ref{TMain} follows by recent work of Mondino-Scharrer \cite[Theorem 1.2]{MondinoScharrer}.  This solves the \emph{existence} portion of the \emph{Canham Problem} \cite{KMS2016MSRI} to characterize the family of $\Wcal$-minimizing surfaces $\{\underline{S}_{g,v}\}$ of genus $g$ with fixed isoperimetric ratio $v \in (0, 1)$. 

\subsection*{Outline of our method} 
\phantom{ab}
\nopagebreak

A naive construction for a smooth surface $S\subset \R^3$ of genus $g$ with $\Wcal(S)$ approximately $8\pi$ and $v(S)$ small would be to join two nearby parallel copies of the unit sphere $\Sph^2 \subset \R^3$, each with $g+1$ disks removed, by gluing in $g+1$ small catenoidal bridges to the boundaries of the excised disks. 
Unfortunately, such a surface has $\Wcal(S) >8\pi$: each catenoidal bridge contributes strictly more bending energy than the energy of the corresponding pair of excised disks (see Lemma \ref{Lbridge}), while each sphere contributes exactly $4\pi$ minus the energy of the excised disks.  Therefore, to ensure the bending energy is less than $8\pi$, a more subtle construction is required.

Our method stems  
in part from the Kapouleas construction \cite{Kap:sphere} 
of minimal surfaces in $\Sph^3$ by doubling the equatorial two-sphere $\Sph^2$.  
He introduced an approach called \emph{Linearized Doubling} (LD), where 
the minimal  
surface 
is a small perturbation of a smooth \emph{initial surface}, constructed by gluing a collection of catenoidal bridges, whose centers are located on a finite set $L \subset \Sph^2$ and whose axes are orthogonal to $\Sph^2$, to the bigraph of an \emph{LD solution}---a solution $\varphi$ of the linearized minimal surface equation $\Lcal \varphi : = (\Delta +2)\varphi = 0$ on $\Sph^2 \setminus L$ with logarithmic singularities at $L$.  In order for a family of LD solutions to give rise to a minimal doubling, certain matching conditions related to the alignment of the LD solutions and the bridges must be satisfied.  While the set of LD solutions is a linear space, these matching conditions are nonlinear and substantial work has gone into determining appropriately matched LD solutions and to further developing doubling constructions via linearized doubling \cite{Kap:sphere, KapMcG, KapMcG2}. 

The surfaces we construct here are defined in similar fashion to the initial surfaces described above.  Importantly, however, the function $\varphi$ whose bigraph defines the surface away from the bridges is now chosen to be a singular solution on $\Sph^2 \setminus L$ of the \emph{linearized Willmore equation} 
$
(\Lcal-2)\Lcal \varphi= \Delta ( \Lcal \varphi) =0.
$
More specifically, we can choose $\varphi = c_1+\tau \Phi$ where $\Phi$ is such an LD solution, $\tau>0$ is sufficiently small but otherwise arbitrary, and $c_1$ is a constant.  The ability to prescribe $c_1$ allows us to rather easily match the asymptotics of $\varphi$ to those of a collection of catenoidal bridges, each with waist radius $\tau$, provided the set $L$ of centers is carefully chosen. 

For each $g\geq 0$ and all $\tau>0$ sufficiently small (depending on $g$), we then define in \ref{dphi} 
a smooth genus $g$ surface $\Sigma_{g+1, \tau} \subset \Sph^3$ whose bending energy can be shown to satisfy $\Wcal(\Sigma_{g+1,\tau})< 8\pi$.  As $\tau \searrow 0$, stereographic projection of $\Sigma_{g+1,\tau}$ converges in the sense of varifolds to the unit sphere $\Sph^2 \subset \R^3$ with multiplicity two.  In particular, since stereographic projection is conformal and $\Wcal$ is conformally-invariant, the stereographic image $S_{g,v(\tau)}$ of $\Sigma_{g+1, \tau}$ in $\R^3$ has bending energy less than $8\pi$ and isoperimetric ratio $v=v(\tau)$ near zero.  Applying a family of M\"obius transformations to $S_{g,v(\tau)}$
we then obtain---for any $v\in[v(\tau),1)$---a genus $g$ surface $S_{g, v}$ with $\Wcal(S_{g, v}) <8\pi$ and isoperimetric ratio $v(S_{g, v}) = v$.  
From the construction of these comparison surfaces and the recent work of Mondino-Scharrer \cite[Theorem 1.2]{MondinoScharrer}, our Theorem \ref{TMain} follows.

\subsection*{Outline of our paper}
\phantom{ab}
\nopagebreak

After fixing  
notation and conventions in Section \ref{Snotation}, we estimate in Section \ref{Sgraph} the bending energy of the normal exponential graph $\Gamma^u_\Omega \subset \Sph^3$ of a function $u$ defined on a domain $\Omega \subset \Sph^2$.  In Section \ref{Sbridge}, we define for each $p \in \Sph^2$ and all small enough $\tau>0$ a catenoidal bridge $\Gcat \subset \Sph^3$ centered at $p$ and of size $\tau$, and we  
estimate the bending energy of a $\Gcat$. 
 
In Section \ref{SLD}, after recalling certain facts from \cite{Kap:sphere} regarding LD solutions, we 
construct in \ref{dphi} for each $m \in \N$ and all sufficiently small $\tau>0$ a smoothly embedded genus $m-1$ surface $\Siginit \subset \Sph^3$.  Finally, we prove Theorem \ref{TMain} by reducing the problem to showing $\Wcal(\Sigma_{m, \tau})<8\pi$, which we do by combining earlier estimates on the bending  
energies of the corresponding catenoidal and graphical regions.

\subsection*{Acknowledgements}
$\phantom{ab}$
\nopagebreak

We thank Nikos Kapouleas, whose constructions of minimal surfaces via linearized doubling inspired this approach to solving the Canham existence problem. 

\section{Notation and conventions} 
\label{Snotation}

Throughout the paper, 
let $\Sph^3 \subset \R^4$ be the unit three-sphere and $\Sph^2 \subset \Sph^3$ 
the equatorial two-sphere defined by $\Sph^2 = \Sph^3 \cap \{x^4 = 0\}$, where $(x^1, x^2, x^3, x^4)$ are the standard coordinates on $\R^4$. 

We define the Willmore bending energy $\Wcal(\Sigma)$ of a (not necessarily closed) surface $\Sigma \subset \Sph^3$ by
\begin{align*}
\Wcal(\Sigma) = |\Sigma| + \frac{1}{4} \int_\Sigma H^2 d\sigma,
\end{align*}
where $d\sigma$ is the area form on $\Sigma$, and where $H$ is its mean curvature in $\Sph^3$.  

\begin{remark}
\label{Rconf}
Note that if $S\subset \R^3$ is a stereographic image of $\Sigma \subset \Sph^3$, then $\Wcal(\Sigma) = \Wcal(S)$, where $\Wcal(S)$ was defined in \eqref{EW}.
\end{remark}

\begin{notation}
For any $X \subset \Sph^2$, we write $\dbold_X$ for the distance from $X$, and define the $\delta$-neighborhood of $X$ by
\[
D_X(\delta) : = \{ p \in \Sph^2 : \dbold_X(p) < \delta\}.
\]
If $X$ is finite we just list its points;   for example, $\dbold_q(p)$ is the geodesic distance between $p$ and $q$ and $D_q(\delta)$ is the geodesic disc in $\Sph^2$ with center $q$ and radius $\delta$. 
\end{notation}

\begin{notation}
We denote by $\Lcal = \Delta +2$ and $(\Lcal-2)\Lcal = \Delta\Lcal = \Delta(\Delta +2)$ the area-Jacobi and $\Wcal$-Jacobi operators of $\Sph^2 \subset \Sph^3$, where $\Delta$ denotes the Laplacian on $\Sph^2$. 
\end{notation}

\subsection*{Cutoff functions}
\phantom{ab}
\nopagebreak

The following notation regarding cutoff functions is standard in gluing constructions  \cite{Kap:sphere}.

\begin{definition}
\label{DPsi} 
We fix a smooth function $\Psi:\R\to[0,1]$ with the following properties:
\begin{enumerate}[label=\emph{(\roman*)}]
\item $\Psi$ is nondecreasing.
\item $\Psi\equiv1$ on $[1,\infty)$ and $\Psi\equiv0$ on $(-\infty,-1]$.
\item $\Psi-\frac12$ is an odd function.
\end{enumerate}
Given $a,b\in \R$ with $a\ne b$,
we define smooth functions
$\psicut[a,b]:\R\to[0,1]$
by
\begin{equation*}
\psicut[a,b]:=\Psi\circ L_{a,b},
\end{equation*}
where $L_{a,b}:\R\to\R$ is the linear function defined by the requirements $L(a)=-3$ and $L(b)=3$.
\end{definition}

Note that $\psicut[a,b]$ has the following properties:
\begin{enumerate}[label={(\roman*)}]
\item $\psicut[a,b]$ is weakly monotone.

\item 
$\psicut[a,b]=1$ on a neighborhood of $b$ and 
$\psicut[a,b]=0$ on a neighborhood of $a$.

\item $\psicut[a,b]+\psicut[b,a]=1$ on $\R$.
\end{enumerate}

Suppose now we have functions $f_0,f_1$, and $\rho$ defined on some domain $\Omega$.
We define a new function 
\begin{equation}
\label{EPsibold}
\Psibold\left [a,b; \rho  \right](f_0,f_1):=
\psicut[a,b ]\circ \rho \, f_1
+
\psicut[b,a]\circ  \rho \, f_0.
\end{equation}
Note that
$\Psibold[a,b;\rho ](f_0,f_1)$
depends linearly on the pair $(f_0,f_1)$
and transits from $f_0$
on $\Omega_a$ to $f_1$ on $\Omega_b$,
where $\Omega_a$ and $\Omega_b$ are subsets of $\Omega$ which contain
$\rho^{-1}(a)$ and $\rho^{-1}(b)$ respectively,
and are defined by
$$
\Omega_a=\rho^{-1}\left((-\infty,a+\frac13(b-a))\right),
\qquad
\Omega_b=\rho^{-1}\left((b-\frac13(b-a),\infty)\right),
$$
when $a<b$, and 
$$
\Omega_a=\rho^{-1}\left((a-\frac13(a-b),\infty)\right),
\qquad
\Omega_b=\rho^{-1}\left((-\infty,b+\frac13(a-b))\right),
$$
when $b<a$.
Clearly if $f_0,f_1,$ and $\rho$ are smooth then
$\Psibold[a,b;\rho ](f_0,f_1)$
is also smooth.

\section{Normal Graphs and their Willmore energy}
\label{Sgraph}

Given a smooth function $u$ defined on a domain $\Omega \subset \Sph^2$, we define its \emph{normal graph} over $\Omega$ by
\begin{align*}
\Gamma^u_\Omega = \left\{ \exp_p ( u(p) \nu(p)) : p\in \Omega\right\}\subset \Sph^3,
\end{align*}
where $\nu$ is the unit normal field on $\Sph^2 \subset \Sph^3$ given by $\nu(p) = (0,0,0, 1)$ for each $p\in \Sph^2$.  If $|u| < \pi/2$, the map $E_u : \Omega \rightarrow \Gamma^u_\Omega$ defined by $E_u(p) = \exp_p(u(p) \nu(p))$ is a diffeomorphism, and a short calculation shows the pullback of the area form  $d\sigma_u$ on $\Gamma^u_\Omega$ is 
\begin{align}
\label{Eda}
E^*_u d\sigma_u = \cos^2(u) \sqrt{1+ \sec^2(u) |\nabla u|^2} d\sigma,
\end{align}
where $d\sigma : = d\sigma_0$ is the Riemannian area form of $\Sph^2$.

Since the surfaces we construct are built from gluing together normal graphs, the following estimate of the bending energy of such a graph will be useful. 
\begin{lemma}
\label{LWomega}
Suppose $\Omega \subset \Sph^2$ is a domain with smooth boundary, $u \in C^2(\Omega)$, and that $\| u\|_{C^2(\Omega)} < \epsilon$ for a given small $\epsilon> 0$.  Then
\begin{align*}
\Wcal(\Gamma^u_\Omega) = |\Omega| + \frac{1}{4}\int_{\Omega} (\Delta u)(\Lcal u) \, d\sigma + \frac{1}{2} \int_{\partial \Omega} u \frac{\partial u}{\partial \eta}ds + O(\epsilon^4).
\end{align*}
\end{lemma}
\begin{proof}
Expanding \eqref{Eda} using the smallness of $u$ and $|\nabla u|$, we have
\begin{align*}
E^*_u d\sigma_u = \left(  1-u^2 + \frac{1}{2}|\nabla u|^2 + O(\epsilon^4) \right) \, d\sigma.
\end{align*}
Because $\Lcal$ is the linearized mean curvature operator and $\Sph^2\subset \Sph^3$ is totally geodesic, the mean curvature $H$ of $\Gamma^u_\Omega$ satisfies $H = \Lcal u + O(\epsilon^4)$.  In combination with the preceding, we have
\begin{align*}
( \textstyle{\frac{1}{4}} H^2 +1) E^*_u d\sigma_u = \left( 1+ \frac{1}{4}(\Lcal u )^2 - u^2 + \frac{1}{2}|\nabla u|^2 + O(\epsilon^4)\right) d\sigma.
\end{align*} 
The conclusion now follows by integrating over $\Gamma^u_{\Omega}$ and integrating $\int_{\Omega}\frac{1}{2}|\nabla u|^2 d\sigma$ by parts. 
\end{proof}

\section{Catenoidal bridges and their Willmore energy}
\label{Sbridge}
Recall that the top half of a catenoid of size $\tau$ in Euclidean three-space can be written as a radial graph of the function $\phicat :[\tau, \infty)\rightarrow \R$ defined by 
\begin{align}
\label{Epcat}
\phicat(r) = 
\tau \arccosh \frac{r}{\tau} := \tau \left( \log \frac{2r}{\tau} + \log \left( \frac{1}{2} + \frac{1}{2}\sqrt{1- \frac{\tau^2}{r^2}}\right)\right).
\end{align}
For future reference we record that
\begin{align}
\label{Ephicatd}
\phicat'(r) = \frac{\tau}{\sqrt{r^2-\tau^2}}.
\end{align}

\begin{assumption}
\label{Atau0}
We fix a small positive number $\alpha$, for example $\alpha = 1/10$, and assume hereafter that $\tau>0$ is as small as needed in terms of $\alpha$.
\end{assumption}

\begin{definition}[Catenoidal bridges]
\label{dcatb}
Given $p\in \Sph^2$, we define the catenoidal bridge $\Gcat$ to be the union of the graphs of $\pm \phicat \circ \dbold_p$ on the domain $D_p(\tau^\alpha) \setminus D_p(\tau)$. 
\end{definition}

The following estimate on the bending energy of a $\Gcat$ should be compared with the estimate on the area of a portion of a Euclidean catenoid in \cite{Ketover}.

\begin{lemma}
\label{Lbridge}
For any $p\in \Sph^2$ and all sufficiently small $\tau>0$, 
\begin{align*}
\Wcal(\Gcat) \leq  2|D_p(\tau^\alpha)|  + \frac{8\pi}{3} \tau^2 |\log \tau|.
\end{align*}
\end{lemma}
\begin{proof}
Integrating the area form from \eqref{Eda} and estimating, we have
\begin{align*}
\frac{1}{2}|\Gcat| &= 2\pi \int_\tau^{\tau^\alpha} \cos^2 \phicat (r) \sqrt{1+ \sec^2 \phicat(r) (\phicat'(r))^2} \sin r\, dr \\
&\leq C\tau^2 + 2\pi \int_{9\tau}^{\tau^\alpha} \cos^2 \phicat(r) \sqrt{1+ \sec^2 \phicat(r) (\phicat'(r))^2} \sin r\, dr \\
& \leq |D_p(\tau^\alpha)| + C\tau^2 +  \frac{5\pi}{4} \int_{9\tau}^{\tau^\alpha} (\phicat'(r))^2 r \, dr \\
&\leq |D_p(\tau^\alpha)| + C\tau^2 + \frac{5\pi}{4} \tau^2 \int_{9\tau}^{\tau^\alpha} \frac{1}{r}  + \frac{C\tau^2}{r^3} \, dr \\
&\leq  |D_p(\tau^\alpha)| + C\tau^2+\frac{5\pi}{4} \tau^2 |\log \tau|,
\end{align*}
where we have used \eqref{Ephicatd} to estimate $(\phicat'(r))^2$  on $(9 \tau, \tau^\alpha)$ as follows:
\begin{align*}
 (\phicat'(r))^2 = \frac{\tau^2}{r^2 - \tau^2} 
 \leq \frac{\tau^2}{r^2} + C \frac{\tau^4}{r^4}.
 \end{align*}
Finally, from \eqref{EHcat} in Appendix \ref{SHbridge}, we have the estimate $|H| \leq C\tau |\log \tau|$ on $\Gcat$.  In combination with the preceding, this implies
\begin{align*}
|\Wcal( \Gcat)| &\leq 
 2|D_p(\tau^\alpha)| + C\tau^2+\frac{5\pi}{2} \tau^2 |\log \tau| + C\tau^{2(1+\alpha)}|\log \tau|^2\\
 &\leq 2 |D_p(\tau^\alpha)|+ \frac{8\pi}{3} \tau^2|\log \tau|.
\end{align*}
\end{proof}

\begin{remark}
Although we will not need it, it is possible to prove the following strengthening of the estimate in Lemma \ref{Lbridge}: $\Wcal(\cat_{p, \tau}) = 2| D_p(\tau^\alpha)| + 2\pi \tau^2 \log (2 \tau^{\alpha-1}) - \pi \tau^2 + O(\tau^{2(1+\alpha)} |\log \tau|^2)$.
\end{remark}

\section{Proof of the Main Theorem}
\label{SLD}
In order to study singular solutions of the equation $\Delta (\Lcal \varphi) = 0$, we first recall (cf. \cite[Lemma 2.20]{Kap:sphere}) the Green's function for $\Lcal$ on $\Sph^2$. 

\begin{lemma}
\label{LG}
The function $G \in C^\infty((0, \pi))$ defined by
\begin{align*}
G(r) : = \cos r \log \left( 2 \tan \textstyle{\frac{r}{2}}\right)+ 1 - \cos r 
\end{align*}
has the following properties:
\begin{enumerate}[label=\emph{(\roman*)}]
\item For each $p\in \Sph^2$, we have $\Lcal ( G \circ \dbold_p) = 0$ on $\Sph^2\setminus \{p, -p\}$. 
\item For small $r>0$, $G(r) = (1+O(r^2))\log r$. 
\item The following estimate holds for $r\in (\tau^\alpha, 9\tau^\alpha)$ and $k \in \{0, 1, 2\}$:
\begin{align*}
\left| \frac{d^k}{dr^k} ( \phicat(r) - \tau G(r) + \tau \log\frac{\tau}{2} \cos r) \right| \leq C \tau^{1+(2-k)\alpha} |\log \tau|.
\end{align*}
\end{enumerate}
\end{lemma}
\begin{proof}
Items (i) and (ii) are straightforward calculations, see e.g. \cite[Lemma 2.20]{Kap:sphere}.  Next, we calculate 
\begin{align*}
G(r) - \log \frac{2r}{\tau} - \log \frac{\tau}{2} \cos r = 
(1-\cos r) ( 1+ \log \frac{\tau}{2r}) + \cos r \log \left( \frac{2}{r} \tan \frac{r}{2}\right).
\end{align*}
It follows from this that for $r\in ( \tau^\alpha, 9\tau^\alpha)$ and $k \in \{0, 1, 2\}$,
\begin{align*}
\left| \frac{d^k}{dr^k} ( G(r)- \log\frac{2r}{\tau}-  \log\frac{\tau}{2} \cos r) \right| \leq C \tau^{(2-k)\alpha} |\log \tau|.
\end{align*}
The conclusion follows from this and \eqref{Epcat}.
\end{proof}

In the remainder of this section, $m\in \N$ will denote a given natural number.  For simplicity of notation, we will suppress the dependence of various constants on $m$. 

\begin{assumption}
\label{Atau1}
We assume hereafter that $\tau>0$ is as small as needed in terms of $m$. 
\end{assumption}

We first define the set $L \subset \Sph^2$ of centers of the bridges in the construction, or equivalently the set of logarithmic singularities of $\varphi[m, \tau]$ defined later in \ref{dvarphi}.

\begin{definition}
\label{dL}
Let $L = L[m]: = \big\{ ( \cos \textstyle{ \frac{2\pi k}{m}} , \sin \frac{2\pi k}{m}, 0, 0)\in \Sph^2: k=1, \dots, m\big \}$.
\end{definition}

We next consider a particular discrete family of LD solutions which were studied in \cite[Def. 6.1]{Kap:sphere}:

\begin{lemma}
\label{LPhi}
For $m \geq 2$, there is a function $\Phi = \Phi[m]$ uniquely determined by the following:
\begin{enumerate}[label=\emph{(\roman*)}]
\item $\Phi \in C^\infty(\Sph^2 \setminus L)$ and $\Lcal \Phi = 0$ on $\Sph^2 \setminus L$. 
\item $\Phi - \log \dbold_L$ is bounded on some deleted neighborhood of $L$ in $\Sph^2$.
\item $\Phi$ is invariant under the group of isometries of $\Sph^2$ which fix $L[m]$ as a set. 
\end{enumerate}
\end{lemma}
\begin{proof}
This follows immediately from \cite[Lemma 3.10]{Kap:sphere}.
\end{proof}

\begin{remark}
\label{Rm2}
When $m=2$, $\Phi$ depends only on $\dbold_L$; in particular, (cf. \cite[Def. 2.18]{Kap:sphere}) we have 
 $\Phi = G\circ \dbold_L +(1-\log 2)\cos \dbold_L= 1+ \cos \dbold_L \log \tan \frac{\dbold_L}{2}$.
\end{remark}

Although there is no function $\Phi[m]$ satisfying \ref{LPhi}(i)-(iii) when $m=1$, the following definition will be sufficient for our later applications in that case.

\begin{definition}
\label{dPhi}
Let $\Phi = \Phi[m]$ be as in \ref{LPhi} for $m\geq 2$ and define $\Phi[m]$ for $m=1$ by
\begin{align*}
\Phi = \Psibold[ \textstyle{\frac{\pi}{3}}, \frac{\pi}{2}; \dbold_L]( G\circ \dbold_L , 0) .
\end{align*}
\end{definition}

Before we define the surfaces $\Siginit$ used in Theorem \ref{TMain}, we need to extract from $\Phi$ the dominant singular and constant parts in the vicinity of $L$.

\begin{lemma}
\label{Ldecomp}
There is a unique $c_0 = c_0[m] \in \R$ such that the function $\Phip$ defined by the decomposition
\begin{align}
\label{Ephip}
\Phi = G \circ \dbold_L + c_0 \cos \dbold_L + \Phip  \quad \text{on} \quad D_L(\delta), \quad
\delta: = 1/(10m)
\end{align}
satisfies $\Phip(p) =0$ for each $p\in L$.  Moreover, $\Phip \in C^\infty(D_L(\delta))$ and $d_p \Phip = 0$ for each $p\in L$. 
\end{lemma}
\begin{proof}
For any $c_0 \in \R$, it follows from \ref{LG} and \ref{LPhi} that $\Phip$ as defined by \eqref{Ephip} is bounded on $D_L(\delta)\setminus L$ and satisfies $\Lcal \Phip = 0$ there.  This implies the smoothness of $\Phip$ on $D_L(\delta)$ by standard removable singularity results.  The symmetries imply that $d_p \Phip = 0$ for each $p\in L$ and $\Phip(p)$ is independent of the choice of $p\in L$.  There is therefore a unique $c_0 \in \R$ such that $\Phip(p) =0$ for all $p\in L$.
\end{proof}

\begin{corollary} For $k\in \{0, 1, 2\}$, $\Phip$ satisfies  $|\nabla^{(k)} \Phip| \leq C ( \dbold_L)^{2-k}$ on $D_{L}(\delta)$. 
\label{Ephipest}
\end{corollary}

\begin{remark}
When $m=1$, note that $c_0 = 0$ and $\Phip \equiv 0$.  When $m=2$, we have via \ref{Rm2} that $c_0 = 1-\log 2$ and $\Phip \equiv 0$.
\end{remark}

\begin{definition}
\label{dvarphi}
Given $\tau>0$, define $\varphi = \varphi[m, \tau] \in C^\infty(\Sph^2 \setminus L)$ by
\begin{align*}
\varphi = 
c_1 + \tau \Phi,
\end{align*}
where $L$ and $\Phi$ are as in \ref{dL} and \ref{dPhi}, $c_1: = \tau \log(2/\tau)-\tau c_0$,  and $c_0$ is as in \ref{Ldecomp}. 
\end{definition}

We are now ready to define the family of surfaces used in Theorem \ref{TMain}.

\begin{definition}
\label{dphi} 
Given $\tau>0$ as in \ref{Atau1},  define the smooth surface $\Siginit \subset \Sph^3$ to be the union over $\Sph^2 \setminus D_L(\tau)$ of the graphs of $\pm \phigl$, where $\phigl: \Sph^2 \setminus D_L(\tau) \rightarrow [0, \infty)$ is defined as follows:
\begin{enumerate}[label=\emph{(\roman*)}]
\item On $\Sph^2 \setminus D_L(2\tau^\alpha)$ we have $\phigl = \varphi[m, \tau]$. 

\item On $D_L(2 \tau^\alpha)\setminus D_L(\tau)$ we have (recall \ref{EPsibold})
	\begin{align*}
	\phigl = \Psibold[ \tau^\alpha, 2 \tau^\alpha; \dbold_L] ( \phicat \circ \dbold_L, \varphi[m, \tau]).
	\end{align*}
\end{enumerate}
Finally, we define $\Omega: = \Sph^2 \setminus D_L(\tau^\alpha)$, so that (recall \ref{dcatb}) $\Siginit =( \bigcup_{p\in L} \cat_{p, \tau} )\cup \Gamma^{\phigl}_{\Omega} \cup \Gamma^{-\phigl}_{\Omega}$.
\end{definition}

\begin{lemma}
\label{Lext}
For all sufficiently small $\tau>0$,
\begin{align*}
\Wcal( \Gld) \leq   |\Sph^2| - | D_L(\tau^\alpha)| - \frac{11m\pi}{6} \tau^2 |\log \tau|. 
\end{align*}
\end{lemma}
\begin{proof} 
We first prove the inequality in the case where $m\geq 2$.  
Since $\phigl = \varphi = c_1 + \tau \Phi$ on $\Omega \setminus A_L$, where $A_L : = D_L(2\tau^\alpha) \setminus D_L(\tau^\alpha)$ is the gluing region,
\begin{align*} 
\int_{\Omega}( \Delta \phigl)( \Lcal \phigl) d\sigma &= \int_{A_L}  (\Delta \phigl)( \Lcal \phigl) d\sigma + 2c_1\tau \int_{\Omega \setminus A_L} \Delta \Phi d\sigma\\
&=  \int_{A_L}  (\Delta \phigl)( \Lcal \phigl) d\sigma + 2 c_1\tau \int_{\partial (\Omega \setminus A_L)} \frac{\partial \Phi}{\partial \eta} ds. 
\end{align*}
It follows from \ref{Ldecomp}, \ref{dvarphi}, and \ref{dphi} that
\begin{align*}
\phigl = \tau G\circ \dbold_L - \tau \log \frac{\tau}{2} \cos \dbold_L +
\Psibold[ \tau^\alpha, 2\tau^\alpha; \dbold_L]( \varphiunder, \varphiover),
\end{align*}
on $A_L$, where
\begin{equation*}
\begin{aligned}
\varphiunder & := \phicat \circ \dbold_L - \tau G \circ \dbold_L + \tau \log (\tau/2) \cos \dbold_L, \\
\varphiover &:= c_1 (1- \cos \dbold_L) + \tau \Phip.
\end{aligned}
\end{equation*}
Therefore, we have on $A_L$
\begin{align*}
\Delta \phigl &= -2 \tau ( G\circ \dbold_L - \log (\tau/2) \cos \dbold_L) + \Delta ( \Psibold[ \tau^\alpha, 2\tau^\alpha; \dbold_L] (\varphiunder, \varphiover)), \\
\Lcal \phigl &= \Lcal  \Psibold[ \tau^\alpha, 2\tau^\alpha; \dbold_L] (\varphiunder, \varphiover),
\end{align*}
and moreover from Lemma \ref{LG}(iii) and \eqref{Ephipest} that $|\Delta \phigl| \leq C \tau | \log \tau|$ and $|\Lcal \phigl| \leq C \tau |\log \tau|$.  Hence,
\begin{align}
\label{Ell}
\int_{\Omega}( \Delta \phigl)( \Lcal \phigl) d\sigma = 
2 c_1\tau \int_{\partial (\Omega \setminus A_L)} \frac{\partial \Phi}{\partial \eta} ds + O(\tau^{2(1+\alpha)} |\log \tau|^2).
\end{align}

We now apply Lemma \ref{LWomega}. First, using \ref{LG} and Lemma \ref{Ldecomp}, it is straightforward to check that $\| \phigl \|_{C^2(\Omega)} \leq C \tau^{1-2\alpha}$.  Working \eqref{Ell} into \ref{LWomega} establishes
\begin{align*}
\Wcal( \Gld) &= |\Sph^2| - |D_L(\tau^\alpha)|+
\frac{1}{2} \int_{\partial \Omega} \phigl \frac{\partial \phigl}{\partial \eta}ds +
\frac{1}{2} c_1 \tau\int_{\partial (\Omega \setminus A_L)} \frac{\partial \Phi}{\partial \eta} \, ds\\
&\phantom{=} + O(\tau^{2(1+\alpha)}| \log \tau|^2)+ O(\tau^{4(1-2\alpha)}). 
\end{align*}
The desired inequality follows from this by using that $\phigl=\phicat\circ \dbold_L$ on a neighborhood of $\partial \Omega$ and estimating the integrals using \eqref{Epcat} and \eqref{Ephicatd}, Lemma \ref{LG}, and Lemma \ref{Ldecomp}.

We now consider the case where $m=1$.  Since $\Lcal \Phi$ is now supported on $A_1 : =  D_L(\frac{\pi}{2})\setminus D_L(\frac{\pi}{3})$ (recall \ref{dPhi}),
\begin{align*}
\int_{\Omega}( \Delta \phigl)( \Lcal \phigl) d\sigma &=  \int_{A_L}  (\Delta \phigl)( \Lcal \phigl) d\sigma +  2 c_1\tau \int_{\partial (\Omega \setminus A_L)} \frac{\partial \Phi}{\partial \eta} ds+ \tau^2 \int_{A_1}( \Lcal \Phi )(\Delta \Phi) d\sigma,
\end{align*}
where $A_L$ is as before.  Estimating the integral over $A_L$ in exactly the same way as before, we find
\begin{align*}
\int_{A_L}  (\Delta \phigl)( \Lcal \phigl) d\sigma = O(\tau^{2(1+\alpha)} |\log \tau|^2). 
\end{align*} 
Next, we have $|\Lcal \Phi| \leq C$ and $|\Delta \Phi| \leq C$ on $A_1$, so $\tau^2 \int_{A_1}( \Lcal \Phi )(\Delta \Phi) d\sigma =O( \tau^2)$.  The proof is completed by working the preceding into Lemma \ref{LWomega} and estimating as before. 
\end{proof}

We are now ready to prove the main result of this paper:

\begin{proof}[Proof of Theorem \ref{TMain}]
By recent work of Mondino-Scharrer \cite[Theorem 1.2]{MondinoScharrer}, it suffices to show, for each $g\in \{0, 1, \dots\}$ and $v\in (0, 1)$, that there exists a smoothly embedded surface $S= S_{g, v} \subset \R^3$ of genus $g$ and isoperimetric ratio $v(S) = v$ and $\Wcal(S) < 8\pi$.  In fact, if for any small $v\in (0, 1)$ we can construct such a surface, then by using conformal invariance of the bending energy $\Wcal(S)$ and applying a family of M\"obius transformations to $S \subset \R^3$ which dilate out from a point $p_+\in S$ and contract in to some other point on $p_-\in S$, we thereby obtain a family of surfaces with the same value of the bending energy $\Wcal$ and all larger isoperimetric ratios $v<1$, since this family of surfaces converges smoothly away from  $p_-$ to a round sphere (with $v=1$). 

Now fix $m\in \N$.  For all small enough $\tau>0$, define $\Sigma_{m, \tau}$ as in \ref{dphi} to be a smoothly embedded genus $g=m-1$ surface in $\Sph^3$.  Let $Y: \Sph^3 \setminus \{(0,0, 0, 1)\} \rightarrow \R^3=\{x_4=0\}$ be stereographic projection from $(0, 0, 0, 1) \in \Sph^3 \subset \R^4$.  Since $\Sigma_{m, \tau}$ converges in the sense of varifolds to $2\Sph^2$ as $\tau \searrow 0$, since $Y$ is conformal, and $\left. Y \right|_{\Sph^2} = \mathrm{Id}_{\Sph^2}$, where we regard $\Sph^2=\R^3\cap\Sph^3$ as both a subset of $\R^3$ and of $\Sph^3$, it follows that $\lim_{\tau \searrow 0} v(Y(\Sigma_{m, \tau})) = 0$.  Thus to prove the theorem, by Remark \ref{Rconf} it suffices to show that $\Wcal(\Sigma_{m, \tau}) < 8\pi$ for all $\tau$ small enough.  To do this, we combine the estimates in \ref{Lbridge} and \ref{Lext}:
\begin{align*}
\Wcal(\Sigma_{m, \tau}) &= m \Wcal( \Gcat) + 2 \Wcal(\Gld)\\
&\leq 2 |D_L(\tau^\alpha)| + \frac{8\pi}{3}m\tau^2 |\log \tau| + 
2|\Sph^2| - 2|D_L(\tau^\alpha)|-\frac{11 \pi}{3}m  \tau^2 | \log \tau|\\
&= 2|\Sph^2| - m \pi  \tau^2 |\log \tau|.
\end{align*}
\end{proof}

\begin{remark}
\label{RL}
The method used here leads to the construction of other comparison surfaces with $\Wcal<8\pi$ and bridges centered on other symmetric configurations of points.  One such configuration consists of the vertices of a regular tetrahedron.
It would be interesting to know whether Canham minimizers for
 a given genus and small isoperimetric ratio have any particular symmetries. 
\end{remark}

\begin{remark}
While $\Phi[m]$ and $\Sigma_{m, \tau}$ have nontrivial symmetries for $m>1$, it is possible to construct comparison surfaces with $\Wcal <8\pi$ and $m>1$ bridges centered on configurations of points $L$ with trivial symmetry by generalizing the construction of $\Phi[1]$ and cutting off $G\circ \dbold_L$ to zero.  However, the resulting function would no longer satisfy the linearized Willmore equation, and we would not generally expect such comparison surfaces to be close to local minimizers for the Canham problem. 
\end{remark}

\appendix
\section{Mean curvature on a catenoidal bridge}
\label{SHbridge}

In this appendix, we estimate the mean curvature of a small catenoidal bridge in $\Sph^3$.  This was done in \cite[Example 2.15]{KapMcG2}, but we summarize the argument in order to keep the exposition self-contained.

Define a parametrization $E :(0, \pi) \times (-\pi, \pi ) \times (-\frac{\pi}{2},\frac{\pi}{2} ) \rightarrow \Sph^3$ by
\begin{align*}
E( \rr, \theta, \zz) = & (  \sin \rr \cos \theta \cos \zz,  \sin \rr \sin \theta \cos\zz, \cos \rr  \cos \zz, \sin \zz).
 \end{align*}
 We take $(\rr, \theta, \zz)$ as local coordinates for $\Sph^3$.  
The pullback metric is
 \begin{align*}
 E^* g = \cos^2 \zz\left( d \rr^2 + \sin^2 \rr d \theta^2\right) + d\zz^2,
\end{align*}
and the only nonvanishing Christoffel symbols are 
\begin{equation}
\label{Echr}
\begin{gathered}
\Gamma_{\rr \zz}^{\rr}  = \Gamma_{\zz \rr}^{\rr}= \Gamma_{\theta \zz}^{\theta} = \Gamma_{\zz \theta}^{\theta} =   - \tan \zz, \quad
\Gamma_{\theta \theta}^\rr = - \sin \rr \cos \rr, \\
\Gamma_{\rr\theta}^{\theta} = \Gamma_{\theta \rr}^{\theta} = \cot \rr, \quad
\Gamma_{\rr\rr}^{\zz} = \cos \zz \sin \zz, \quad
\Gamma_{\theta \theta}^{\zz} = \sin^2 \rr \sin \zz \cos \zz .
\end{gathered}
\end{equation} 
Define a map $X: [-\sunder, \sunder] \times (-\pi, \pi) \rightarrow \R^3$ by 
\begin{equation}
\label{EXcat}
\begin{aligned}
X(\sss, \vartheta) &= (\rr(\sss, \vartheta), \theta(\sss, \vartheta), \zz(\sss, \vartheta)) \\
&= (\tau \cosh \sss, \vartheta, \tau \sss),
\end{aligned}
\end{equation}
where $\sunder$ is defined by the equation $\tau \cosh \sunder = 9 \tau^\alpha$. 

Calculation shows that the pullback metric in $(\sss, \vartheta)$ coordinates is
\begin{equation}
\label{Ebrfermi}
X^* E^* g = \rr^2( 1- \tanh^2 \sss \sin^2 \zz) d\sss^2 + \cos^2 \zz \sin^2 \rr \, d\vartheta^2,
\end{equation}
where $\rr = \rr(\sss, \vartheta)$ and $\zz = \zz(\sss, \vartheta)$, 
and that
\begin{align*}
\nu =( \tanh \sss\,  \partial_{\zz} - \sec^2\zz \sech \sss \, \partial_{\rr})/\sqrt{1+\tan^2 \zz \sech^2 \sss}
\end{align*}
is a unit normal field along the image of $X$. 
We compute the second fundamental form $A$ of $X$ using the formula
$A = \left( X^{k}_{, \alpha \beta} + \Gamma_{lm}^{k} X^l_{, \alpha} X^{m}_{, \beta}\right) g_{kn}\nu^n dx^\alpha dx^\beta$, where $X$ is as in \ref{EXcat}, we have renamed the cylinder coordinates $(x^1, x^2): = (\sss, \vartheta)$, and Greek indices take the values $1$ and $2$ while Latin indices take the values $1,2,3$, corresponding to the coordinates $\rr, \theta, \zz$.  Using the preceding and the Christoffel symbols in \eqref{Echr}, we find
\begin{align*}
A &= \frac{ \left[ \tau^2 \tanh \sss \left( \tan \zz + \frac{1}{2}\sinh^2 \sss \sin 2\zz \right) - \tau\right] d\sss^2 +\frac{1}{2}(\sin 2\rr \sech \sss + \sin^2 \rr \sin 2\zz \tanh \sss) \, d\vartheta^2 }{ \sqrt{ 1+ \tan^2 \zz \sech^2 \sss}}\\
&= (1+ O(\zz^2)) \left( \tau( - d\sss^2 + d\vartheta^2) +  O(\rr^2 \zz)d\sss^2 + O(\rr^3+\rr^2 \zz) d\vartheta^2\right),
\end{align*}
where in the second equality we have estimated using that $ \sqrt{ 1+ \tan^2 \zz \sech^2 \sss} = 1+ O(\zz^2)$ and that $\frac{1}{2}\sin 2\rr \sech \sss = \tau + O(\rr^3)$.  Finally, using that $\rr^2 g^{\sss \sss} = 1+ O(\zz^2)$ and $\rr^2 g^{\vartheta \vartheta} = 1+O(\zz^2 +\rr^2)$ we estimate
\begin{align}
\label{EHcat}
\rr^2 H = O\left(\tau \zz^2 +\rr^2 |\zz| +  \tau \rr^2 \right).
\end{align}

\end{document}